\documentclass{amsart}

\usepackage{fonts, topology, nicetikz, todonotes, mathtools, verbatim, pgfplots, orcidlink}

\usepackage[utf8]{inputenc}

\newtheorem{conjecture}{Conjecture}[section]
\newtheorem{corollary}[conjecture]{Corollary}
\newtheorem{definition}[conjecture]{Definition}
\newtheorem{example}[conjecture]{Example}
\newtheorem{lemma}[conjecture]{Lemma}
\newtheorem{proposition}[conjecture]{Proposition}
\newtheorem{remark}[conjecture]{Remark}
\newtheorem{theorem}[conjecture]{Theorem}

\DeclareMathOperator{\codim}{codim}

\DeclareMathOperator{\Der}{Der}
\DeclareMathOperator{\ev}{ev}
\DeclareMathOperator{\id}{id}

\DeclareMathOperator{\ord}{ord}

\DeclareMathOperator{\mult}{mult}

\title{A proof of the Mond conjecture for wave fronts}
\date{\today}

\author{C. Mu\~noz-Cabello \orcidlink{0000-0002-3115-1851}, J.J. Nu\~no-Ballesteros \orcidlink{0000-0001-6725-541X}, R. Oset Sinha \orcidlink{0000-0002-5652-7982}}

\address{Departament de Matem\`{a}tiques,
Universitat de Val\`encia, Campus de Burjassot, 46100 Burjassot,
Spain}
\email{Christian.Munoz@uv.es}
\email{Juan.Nuno@uv.es}
\email{Raul.Oset@uv.es}

\thanks{Work of C. Mu\~noz-Cabello, Juan J. Nu\~no-Ballesteros and R. Oset Sinha partially supported by Grant PID2021-124577NB-I00 funded by MCIN/AEI/ 10.13039/501100011033 and by ``ERDF A way of making Europe"}


\begin{document}

\maketitle

\begin{abstract}
We prove the Mond conjecture for wave fronts which states that the number of parameters of a frontal versal unfolding is less than or equal to the number of spheres in the image of a stable frontal deformation with equality if the wave front is weighted homogeneous. We give two different proofs. The first one depends on the fact that wave fronts are related to discriminants of map germs and we then use the analogous result proved by Damon and Mond in this context. The second one is based on ideas by Fernández de Bobadilla, Nuño-Ballesteros and Peñafort Sanchis and by Nuño-Ballesteros and Fernández-Hernández. The advantage of the second approach is that most results are valid for any frontal, not only wave fronts, and thus give important tools which may be useful to prove the conjecture for frontals in general.
\end{abstract}

\section{Introduction}

The Mond conjecture is an inequality of type $\mu\ge \tau$, but in the context of the Thom-Mather theory of singularities of mappings. It was stated by D. Mond in \cite{Mond_VanishingCycles} and it says that given a holomorphic map germ $f\colon (\mb C^n,S)\to(\mb C^{n+1},0)$, with $S\subset\mb C^n$ a finite set, with isolated instability and such that $(n,n+1)$ are in the range of nice dimensions of Mather \cite{Mather_VI}, then 
\begin{equation}\label{eq:image}
\codim_{\ms A_e}(f)\le \mu_I(f),
\end{equation}
with equality if $f$ is weighted homogeneous. 

The two invariants of the inequality appear in the scope of $\ms A$-equivalence of mappings, where we allow holomorphic coordinate changes in the source and target. On one hand, the $\ms A_e$-codimension, denoted by $\codim_{\ms A_e}(f)$, 
was introduced by Mather and it measures the minimal number of parameters necessary to have a versal unfolding of $f$. So, it  plays the role of the Tjurina number $\tau$ in the context of deformations of isolated hypersurface singularities (IHS) $(X,0)\subset(\mb C^{n+1},0)$.

On the other hand, $\mu_I(f)$ is called the image Milnor number and was introduced by Mond in \cite{Mond_VanishingCycles} in analogy with the classical Milnor number $\mu$ of an IHS. When $(n,n+1)$ are nice dimensions, $f$ always admits a stabilisation, that is, a 1-parameter unfolding $F=(f_t,t)$ with the property that $f_0=f$ and $f_t$ has only stable singularities for $t\ne0$. If we denote by $X_t$ the image of $f_t$ and by $B_\eps \subseteq \mb{C}^{n+1}$ the closed ball of radius $\eps > 0$, Mond showed that for $0\ll\delta\ll \epsilon\ll 1$ and $0<|t|<\delta$, the space $X_t\cap B_\epsilon$ has the homotopy type of a wedge of $n$-spheres, and the number of such spheres, $\mu_I(f)$, only depends on $f$.

Although the classical $\mu\ge\tau$ inequality for IHS is more or less obvious, the Mond conjecture is known to be true only for plane curves and surfaces (see \cite{Mond_VanishingCycles,Mond_wires}), but it remains open for dimensions $n\ge 3$. We refer to the book \cite{MondNuno} or the recent survey \cite{NunoPenafort} for more details. 

We also remark that there is another version of the inequality for singularities of mappings $f\colon (\mb C^n,S)\to(\mb C^{p},0)$, with $n\ge p$, as follows:
\begin{equation}\label{eq:disc}
\codim_{\ms A_e}(f)\le \mu_\Delta(f),
\end{equation}
with equality if $f$ is weighted homogeneous. Here, $\mu_\Delta(f)$ is the discriminant Milnor number, defined in an analogous way to the image Milnor number, but taking the discriminant $\Delta(f_t)$ (i.e., the image of the critical locus) instead of the image of a stabilisation $f_t$. This inequality is in fact a theorem that was proved by Damon and Mond in \cite{MondDamon}.

In this paper, we consider an inequality of type $\mu\ge\tau$ but in the context of deformations of singularities of frontals. We have developed in the last years this theory (see \cite{FrontalSurfaces,Crk1Frontals}), where we consider the subclass of frontal map germs $f\colon(\mb C^n,S)\to(\mb C^{n+1},0)$. Roughly speaking, $f$ is a frontal if its image in $\mb C^{n+1}$ has a well defined tangent hyperplane at each point $x$ in a neighbourhood of $S$ in $\mb C^n$ and hence, $f$ admits a smooth Nash lift $\tilde f$ to the projectivised cotangent bundle $PT^*\mb C^{n+1}$. The notion of frontal appeared for the first time in the real smooth case in the works of Fujimori, Saji, Umehara and Yamada \cite{FSUY} as a generalisation of a wave front (in the terminology of the Russian school of singularities \cite{Arnold_I}), in the particular case that the Nash lift $\tilde f$ is an immersion.

The frontal structure is preserved by $\ms A$-equivalence, so we can consider the induced equivalence for frontals. However, an unfolding of a frontal is not a frontal in general, hence we have to consider only frontal deformations which preserve such structure. In this setting, we also have a well stablished notion of frontal $\ms F_e$-codimension, denoted by $\codim_{\ms F_e}(f)$, which gives the minimal number of parameters necessary for a versal frontal unfolding. 
Analogously, if $f$ has isolated frontal instability and $(n,n+1)$ are nice dimensions in the frontal sense, then $f$ admits a frontal stabilisation $F=(f_t,t)$ and following the arguments of Mond, we can show that $X_t\cap B_\epsilon$ has the homotopy type of a wedge of $n$-spheres, where $X_t$ is the image of $f_t$, $0\ll\delta\ll \epsilon\ll 1$ and $0<|t|<\delta$. We call the number of such spheres the frontal Milnor number $\mu_\ms{F}(f)$. 

\begin{example}{\rm
Consider the $E_6$-singularity, i.e., the plane curve $(X,0)$, with $X=\{y\in\mb C^2\ |\  y_1^4-y_2^3=0\}$ (Fig. \ref{fig:E6}.(a)). The classical Milnor number is $\mu=6$, since its Milnor fibre $B_\epsilon\cap\{y\in\mb C^2\ |\ y_1^4-y_2^3=t\}$, with $0<|t|<\delta$, has the homotopy type of a wedge of six 1-spheres  (Fig. \ref{fig:E6}.(b)). 

The curve $(X,0)$ can be also seen as the image of the map germ $f\colon(\mb C,0)\to(\mb C^2,0)$, $f(x)=(x^3,x^4)$. The deformation $f_t(x)=(x^3-tx,x^4+(3t/2)x^2)$ defines a stabilisation, whose image is an immersed curve with 3 nodes. This implies that $B_\epsilon\cap f_t(\mb C)\simeq S^1\vee S^1\vee S^1$, with $0<|t|<\delta$, so $\mu_I(f)=3$ (Fig. \ref{fig:E6}.(c)). 

However, $f_t$ is not a frontal stabilisation, since the map $F(x,t)=(f_t(x),t)$ is not a frontal. In fact, a frontal stabilisation is given by $f'_t(x)=(x^3+tx,x^4+(2t/3)x^2)$. The image of $f'_t$ is now a plane curve with 2 cusps and one node. In this case, $B_\epsilon\cap f'_t(\mb C)\simeq S^1$, with $0<|t|<\delta$, hence $\mu_\ms F(f)=1$ (Fig. \ref{fig:E6}.(d)).
\begin{figure}[ht]
\begin{center}
\includegraphics[width=3cm]{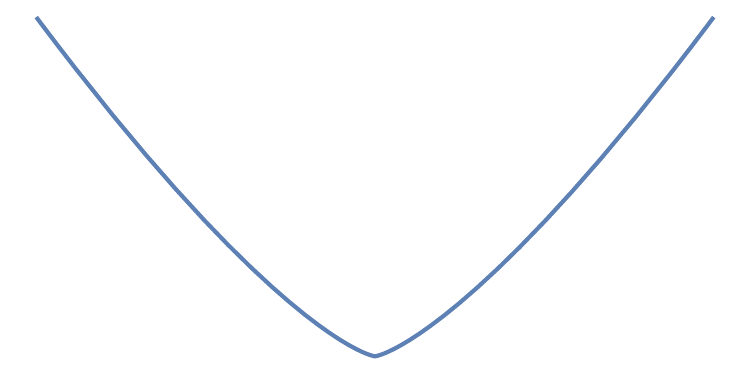}\includegraphics[width=3cm]{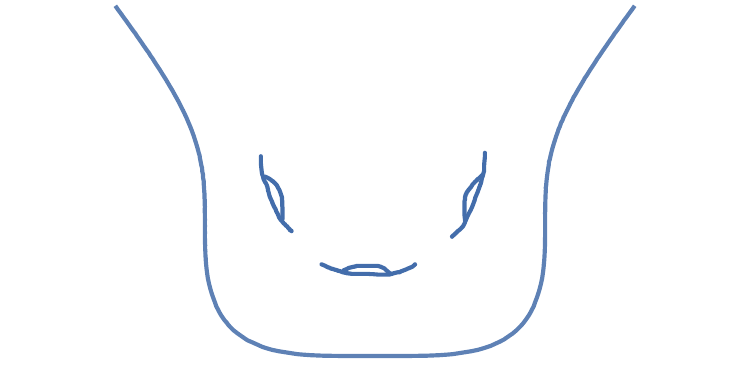}\includegraphics[width=3cm]{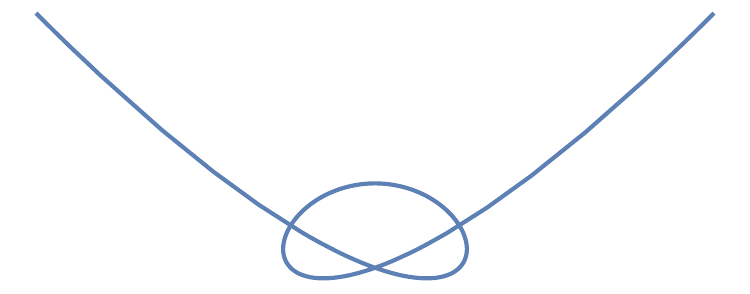}\includegraphics[width=3cm]{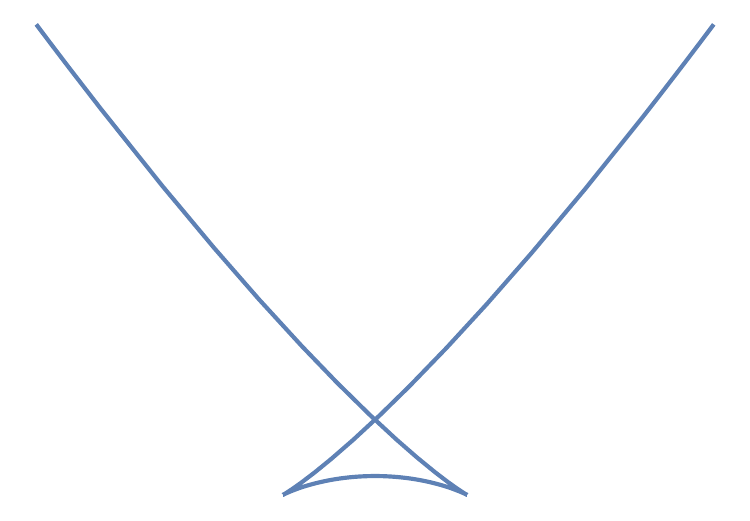}

(a) \hskip2.5cm (b) \hskip2.5cm (c) \hskip2.5cm (d)
\end{center}
\caption{Deformations of the $E_6$-singularity}\label{fig:E6}
\end{figure}
}
\end{example}

It is now natural to ask for a frontal version of the Mond conjecture as follows: let $f\colon (\mb C^n,S)\to(\mb C^{n+1},0)$ be a frontal with isolated frontal instability and such that $(n,n+1)$ are frontal nice dimensions, then 
\begin{equation}\label{eq:frontal}
\codim_{\ms F_e}(f)\le \mu_\ms F(f),
\end{equation}
with equality if $f$ is weighted homogeneous. The frontal Mond conjecture \eqref{eq:frontal} has been already proved for plane curves (i.e., $n=1$) in \cite{Crk1Frontals}. It follows from the original Mond conjecture \eqref{eq:image} and the fact that $\codim_{\ms F_e}(f)=\codim_{\ms A_e}(f)-\mult(f)+1$ and also $\mu_\ms F(f)=\mu_I(f)-\mult(f)+1$, where $\mult(f)$ is the multiplicity of the plane curve.

In this paper, we follow the ideas of \cite{Disentanglements} and define a Jacobian module $M_\ms F(g)$, where $g\colon(\mb C^{n+1},0)\to(\mb C,0)$ is a reduced equation of the image of $f$, with the property that
\[
\dim_\mb C M_\ms F(g)\ge \codim_{\ms F_e}(f),
\]
with equality if $f$ is weighted homogeneous. Hence, the frontal Mond conjecture \eqref{eq:frontal} follows if we could show that $\mu_\ms F(f)=\dim_\mb C M_\ms F(g)$. In fact, we also define a relative version of the Jacobian module, $M_y(G)$, where $G$ a reduced equation of the image of an $r$-parameter frontal unfolding $F$, such that $M_y(G)\otimes \ms O_r\cong M_\ms F(g)$. Furthermore, if $F$ is frontal stable, then
\[
\mu_\ms F(f)=e(\mathfrak m_r;M_y(G)),
\]
the Samuel multiplicity of $M_y(G)$ over the parameter ring $\ms O_r$. As a consequence, when $M_y(G)$ is Cohen-Macaulay of dimension $r$ we get the desired equality $\mu_\ms F(f)=\dim_\mb C M_\ms F(g)$ and hence, the frontal Mond conjecture \eqref{eq:frontal}. In the last part of the paper, we show the frontal Mond conjecture \eqref{eq:frontal} for wave fronts. Our proof follows similar arguments to the proof of the Damon-Mond inequality \eqref{eq:disc} for the discriminant Milnor number in \cite{MondDamon}.

\section{Frontal map germs}
\subsection{Deformations of frontal map germs in corank $1$}
Let $PT^*\mb{C}^{n+1}$ be the projectivized cotangent bundle of $\mb{C}^{n+1}$.
If $(z,[\omega]) \in PT^*\mb{C}^{n+1}$, we equip $PT^*\mb{C}^{n+1}$ with the contact structure given by the differential form
	\[\alpha=\omega_1\,dz^1+\dots+\omega_{n+1}\,dz^{n+1}.\]
We consider the canonical projection $\pi\colon PT^*\mb{C}^{n+1} \to \mb{C}^{n+1}$ given by $\pi(z,[\omega])=z$, whose fibres are Legendrian submanifolds of $PT^*\mb{C}^{n+1}$ under this contact structure (i.e. $\ker d\pi_{(z,[\omega])} \subseteq \ker\alpha_{(z,[\omega])}$ for all $(z,[\omega]) \in PT^*\mb{C}^{n+1}$).
A holomorphic map $F\colon N \subset \mb{C}^n \to PT^*\mb{C}^{n+1}$ is \textbf{integral} if $F^*\alpha=0$.

\begin{definition}
	Let $N \subset \mb{C}^n$ be an open subset.
	A holomorphic map $f\colon N \to \mb{C}^{n+1}$ is \textbf{frontal} if there exists an integral map $F\colon N \to PT^*\mb{C}^{n+1}$ such that
		\[f=\pi \circ F.\]
	If $F$ is an immersion, we say $f$ is a \textbf{Legendrian map} or \textbf{wave front}.
	Similarly, a hypersurface $X \subset \mb{C}^{n+1}$ is \textbf{frontal} (resp. a \textbf{wave front}) if there exists a frontal map (resp. wave front) $f\colon N \to \mb{C}^{n+1}$ such that $X=f(N)$.
\end{definition}

If $F\colon N \to PT^*\mb{C}^{n+1}$ is an integral map and $f=\pi\circ F$,
	\[0=F^*\alpha=\sum^{n+1}_{i=1}\nu_i d(Z_i\circ F)=\sum^{n+1}_{i=1}\sum^n_{j=1}\nu_i\frac{\p f_i}{\p x_j}\,dx^j\]
for some $\nu_1,\dots,\nu_{n+1} \in \ms{O}_n$, not all of them zero.
This is the same as claiming that there exists a nowhere-vanishing differential $1$-form $\nu$ on $f(N)$ such that $\nu(df\circ \xi)=0$ for all vector fields $\xi$ on $N$.

Since $PT^*\mb{C}^{n+1}$ is a fibre bundle, we can find for each $(z,[\omega]) \in PT^*\mb{C}^{n+1}$ an open neighbourhood $Z \subset \mb{C}^{n+1}$ of $z$ and an open $U \subseteq \mb{C}P^n$ such that $\pi^{-1}(Z)\cong Z\times U$.
Therefore, $F$ is contact equivalent to the integral mapping $\tilde f(x)=(f(x),[\nu_x])$, known as the \textbf{Nash lift} of $f$.
In particular, if $\Sigma(f)$ is nowhere dense in $N$, $\nu$ can be uniquely extended from $N\backslash \Sigma(f)$ to $N$, and there is a one-to-one correspondence between $f$ and $\tilde f$.
Such a frontal map is known as a \textbf{proper frontal} map \cite{Ishikawa_Survey}.

\begin{definition}
	Let $S \subset \mb{C}^n$ be a finite set.
	A holomorphic multigerm $f\colon (\mb{C}^n,S) \to (\mb{C}^{n+1},0)$ is \textbf{frontal} if it has a frontal representative $f\colon N \to Z$.
	Given a hypersurface $X \subset \mb{C}^{n+1}$, $(X,0)$ is a \textbf{frontal} hypersurface germ if there exists a finite frontal map germ $f\colon (\mb{C}^n,S) \to (\mb{C}^{n+1},0)$ such that $(X,0)=f(\mb{C}^n,S)$.
\end{definition}

Similarly, we shall say that $f$ is a proper frontal multigerm if it has a representative which is a proper frontal map.

\begin{definition}
	A frontal map germ $f\colon (\mb{C}^n,S) \to (\mb{C}^{n+1},0)$ is a \textbf{germ of wave fronts} if $\tilde f$ is an immersion.
	We also call the image of a germ of wave fronts a \textbf{wave front}.
\end{definition}

Other authors such as \cite{Arnold_I} refer to $f$ as a \emph{Legendrian map germ}, and its image as a \emph{front}.

Given a smooth map germ $f\colon (\mb{C}^n,S) \to (\mb{C}^{n+1},0)$, we define the \textbf{ramification ideal} of $f$ as the ideal $\mc{R}(f) \subseteq \ms{O}_n$ generated by the $n\times n$ minors of the Jacobian matrix of $f$,
\[\begin{pmatrix}
\dfrac{\p f_1}{\p x_1}	&	\dots		& \dfrac{\p f_{n+1}}{\p x_1}	\\
\vdots			&	\ddots	& \vdots 					\\
\dfrac{\p f_1}{\p x_n}	&	\dots		& \dfrac{\p f_{n+1}}{\p x_n}
\end{pmatrix}\] 

\begin{proposition}[\cite{Ishikawa_Survey}]\label{jacobian criterion frontal}
	Let $f\colon (\mb{C}^n,S) \to (\mb{C}^{n+1},0)$ be a smooth map germ with ramification ideal $\mc{R}(f)$.
	Then $f$ is a frontal map germ if and only if $\mc{R}(f)$ is a principal ideal.
\end{proposition}

\begin{example}
	\begin{enumerate}
		\item Let $\gamma\colon (\mb{C},0) \to (\mb{C}^2,0)$ be the analytic plane curve given by $\gamma(x)=(p(x),q(x))$, and assume $\ord q \geq \ord p$.
		We have $\mc{R}(\gamma)=\langle p' \rangle$, so $\gamma$ is a frontal map germ.

		\item The folded Whitney umbrella can be parametrized as
		\begin{funcion*}
			f\colon (\mb{C}^2,0) \arrow[r] & (\mb{C}^3,0)\\
			(x,y) \arrow[r, maps to] & (x,y^2,xy^3)
		\end{funcion*}
		We have $\mc{R}(f)=\langle y\rangle$, so $f$ is a frontal map germ.
		
		\item The $F_4$ singularity from Mond's classification \cite{Mond_Classification} can be parametrised as
		\begin{funcion*}
			f\colon (\mb{C}^2,0) \arrow[r] & (\mb{C}^3,0)\\
			(x,y) \arrow[r, maps to] & (x,y^2,y^5+x^3y)
		\end{funcion*}
		We have $\mc{R}(f)=\langle x^3,y\rangle$, which is not a principal ideal, so $f$ is not a frontal.
	\end{enumerate}
\end{example}

We now state a series of definitions and results that we shall use throughout this paper.
Proofs for these statements can be found in \cite{Crk1Frontals}.

We say that two smooth map germs $f,g\colon (\mb{C}^n,S) \to (\mb{C}^p,0)$ (not necessarily frontal) are $\ms{A}$-equivalent if there are diffeomorphisms $\phi\colon (\mb{C}^n,S) \to (\mb{C}^n,S)$ and $\psi\colon (\mb{C}^p,0) \to (\mb{C}^p,0)$ such that $g=\psi\circ f\circ \phi^{-1}$.
We also define the space
	\[T\ms{A}_ef=\left\{\left.\frac{df_t}{dt}\right|_{t=0}: f_t=\psi_t\circ f\circ \phi_t^{-1}\right\},\]
where $\phi_t\colon (\mb{C}^n,S) \to (\mb{C}^n,S)$ and $\psi_t\colon (\mb{C}^p,0) \to (\mb{C}^p,0)$ are families of diffeomorphisms and $t \in \mb{D}$.
This space can be computed algebraically as the sum $tf(\theta_n)+\omega f(\theta_p)$, where
\begin{funcion} \label{t y omega}
	tf\colon \theta_n \arrow[r] 	& \theta(f) 		&& \omega f\colon \theta_p \arrow[r] 	& \theta(f) \\
	\xi \arrow[r, maps to] 		& df\circ \xi 	&& 	\eta \arrow[r, maps to]		& \eta \circ f
\end{funcion}

\begin{proposition}\label{frontals preserved under A}
	Let $f,g\colon (\mb{C}^n,S) \to (\mb{C}^{n+1},0)$ be $\ms{A}$-equivalent holomorphic map germs.
	If $f$ is frontal, $g$ is frontal.
	Moreover, if $f$ is a germ of wave front, $g$ is a germ of wave front.
\end{proposition}

Given a frontal map germ $f\colon (\mb{C}^n,S) \to (\mb{C}^{n+1},0)$, we define the space of infinitesimal frontal deformations of $f$ as
	\[\ms{F}(f):=\left\{\left.\frac{df_t}{dt}\right|_{t=0}: f_0=f, (f_t,t) \text{ frontal}\right\}\]
Using Proposition \ref{frontals preserved under A}, we see that $T\ms{A}_ef \subseteq \ms{F}(f)$.
In particular, this implies that $tf(\xi)$ and $\omega f(\eta)$ are in $\ms{F}(f)$ for all $\xi \in \theta_n$ and $\eta \in \theta_{n+1}$.

\begin{definition}
	We define the \textbf{frontal codimension} or $\ms{F}$-codimension of $f$ as
		\[\codim_{\ms{F}e}(f):=\dim_\mb{C}\frac{\ms{F}(f)}{T\ms{A}_ef}\]
	We say $f$ is $\ms{F}$-finite if $\codim_{\ms{F}_e}(f) < \infty$.
\end{definition}

Given a proper frontal map-germ $f\colon (\mb{C}^n,S) \to (\mb{C}^{n+1},0)$, we shall say $f$ has \textbf{corank at most $1$} if the maximum corank among its branches is $1$.
We also define the \textbf{integral corank} of $f$ as the corank of its Nash lift, which is defined in a similar fashion.

\begin{definition}
	A \textbf{frontal unfolding} of a frontal map germ $f\colon (\mb{C}^n,S) \to (\mb{C}^{n+1},0)$ is an unfolding $F\colon (\mb{C}^r\times \mb{C}^n,\{0\}\times S) \to (\mb{C}^r\times \mb{C}^{n+1},0)$ which is frontal as a map germ.
	A frontal map germ $f\colon (\mb{C}^n,S) \to (\mb{C}^{n+1},0)$ is stable as a frontal or $\ms{F}$-stable if every $r$-parameter frontal unfolding $F$ of $f$ is $\ms{A}$-equivalent to $f\times\id_{(\mb{C}^r,0)}$.
\end{definition}

\begin{theorem}
	Let $f\colon (\mb{C}^n,S) \to (\mb{C}^{n+1},0)$ be a proper frontal map germ with integral corank at most $1$.
	Then $f$ is stable as a frontal if and only if $\ms{F}(f)=T\ms{A}_ef$.
\end{theorem}

\subsection{Frontal Milnor number} 
We define the frontal Milnor number, by using analogous arguments to those of Mond in \cite{Mond_VanishingCycles}. Let $f\colon (\mb{C}^n,S) \to (\mb{C}^{n+1},0)$ be a frontal map germ.

\begin{definition}
	A \textbf{frontal stabilisation}  of $f$ is a $1$-parameter frontal unfolding $F=(t,f_t)$ of $f$ such that $f_t$ has only stable frontal singularities, for all $t\ne0$ in a neighbourhood of the origin in $\mb{C}$.
\end{definition}

\begin{definition}	
	We say that $(n,n+1)$ are \textbf{frontal nice dimensions} if 
	\begin{enumerate}
	\item Any proper frontal map germ $f\colon (\mb{C}^n,S) \to (\mb{C}^{n+1},0)$ with isolated instabillity admits a frontal stabilisation.
	\item There exist only a finite number of $\ms{A}$-equivalence classes of frontal stable map germs $f\colon (\mb{C}^n,S) \to (\mb{C}^{n+1},0)$.
	\end{enumerate}
\end{definition}

In the real $\ms{C}^\infty$ case, the condition (1) is equivalent to saying that the stable frontal map germs are dense within the space of proper frontal map germs $(\mb{R}^n,S) \to (\mb{R}^{n+1},0)$ and the condition (2) follows from (1).
This notion is the equivalent within the framework of frontal map germs to Mather's nice dimensions (\cite{Mather_V}; see also \cite{MondNuno}, \S 5.2).

From now on, we assume that $(n,n+1)$ are frontal nice dimensions and that $f$ has isolated frontal instability. Let $F=(t,f_t)$ be a stabilisation and denote by $(\mathcal X,0)$ the image of $F$ and by $\pi\colon(\mathcal X,0)\to(\mb{C},0)$ the projection $\pi(t,y)=t$. For each $t$ in a neighbourhood of $0$ in $\mb{C}$, the fibre $X_t=\pi^{-1}(t)$ is the image of $f_t$. We fix in $(\mathcal X,0)$ the stratification by frontal stable types, which is well defined by condition (2) above.
By a theorem due to L\^e \cite{Le:1977}, for all $0<\delta\ll\epsilon\ll1$, the restriction
\[
\pi\colon \mathcal X\cap (D_\delta^*\times B_\epsilon)\longrightarrow D_\delta^*
\]
is a locally $\ms{C}^0$-trivial fibration. The fibre, $X_t\cap B_\epsilon$, with $0<|t|<\delta$, is called the \textbf{frontal disentanglement}. Here, 
\[
B_\epsilon=\{y\in\mb{C}^{n+1}\ |\ \|y\|\le\epsilon\},\quad D_\delta^*=\{t\in\mb{C}\ |\ 0<|t|<\delta\}.
\]

\begin{lemma} With the above notation, $\pi\colon(\mathcal X,0)\to(\mb{C},0)$ has isolated critical point in the stratified sense. In particular, the frontal disentanglement has the homotopy type of a wedge of $n$-spheres.
\end{lemma}

\begin{proof} 
The fact that the fibre $X_t\cap B_\epsilon$, with $0<|t|<\delta$, has the homotopy type of a wedge of $n$-spheres, follows from a result of Hamm and L\^e \cite{Le-Hamm:2020}, once we know that $(\mathcal X,0)$ is a hypersurface and $\pi\colon(\mathcal X,0)\to(\mb{C},0)$ has isolated critical point in the stratified sense.

In order to prove that $\pi\colon(\mathcal X,0)\to(\mb{C},0)$ has isolated critical point we fix a representative $\mathcal X$ such that if $t\ne 0$, $f_t$ has only frontal stable singularities on $X_t$ and $f_0=f$ has only frontal stable singularities on $X_0\setminus\{0\}$. Thus, for all $(t,y)\in\mathcal X\setminus\{0\}$, $f_t$ is frontal stable at $y\in X_t$. Hence, $F$ is a trivial unfolding at $(t,y)$ and there exists a germ of diffeomorphism $\Psi=(t,\psi_t)$ which sends $\mathcal X$ into $\mb C\times X_0$. Since $\Psi$ must preserve the strata, we conclude that $\pi$ is a stratified submersion at $(t,y)$. 
\end{proof}

\begin{definition} The number of spheres of the frontal disentanglement is called the \textbf{frontal Milnor number} and is denoted by $\mu_{\ms{F}}(f)$.
\end{definition}

The frontal Milnor number of $f$ can be computed using a formula due to Siersma:

\begin{theorem}[\cite{Siersma}, Theorem 2.3]
With the above notation, suppose that $G\in\ms{O}_{n+2}$ is a reduced equation of $(\mathcal X,0)$. Then, the number of spheres of the fibre $X_t\cap B_\epsilon$, with $0<|t|<\delta$, is equal to 
		\[\sum_{y \in B_\eps \backslash X_t}\mu(g_t;y),\]
where $g_t(y)=G(t,y)$ and $\mu(g_t;y)$ is the classical Milnor number of $g_t$ at $y$.
\end{theorem}

The notion of frontal Milnor number was first introduced in \cite{FrontalSurfaces}, wherein we formulate a frontal version of the conjecture stated by D. Mond in \cite{Mond_VanishingCycles}:

\begin{conjecture}\label{mond conjecture frontal}
	Let $f\colon (\mb{C}^n,S) \to (\mb{C}^{n+1},0)$ be an $\ms{F}$-finite frontal map germ.
	If $(n,n+1)$ is in the frontal nice dimensions,
		\[\mu_{\ms{F}}(f) \geq \codim_{\ms{F}_e}f,\]
	with equality if $f$ is quasihomogeneous.
\end{conjecture}


Conjecture \ref{mond conjecture frontal} has so far been proven for $n=1$ in \cite{Crk1Frontals}, where it can be derived from an expression for the $\ms{F}_e$-codimension of a plane curve.
Our aim in this article is to prove Conjecture \ref{mond conjecture frontal} for all wave fronts within the range of the nice dimensions.

\begin{remark} \label{conjecture Legendrian}
	Assume that the map germ $f$ in Conjecture \ref{mond conjecture frontal} is a proper germ of wave front.
	A $k$-parameter unfolding $F=(f_u,u)$ of $f$ is frontal if and only if $f_t$ lifts into an integral deformation of $\tilde f$ for all $t$ in an open neighbourhood $U \subseteq \mb{C}^k$ of $0$ (\cite{Crk1Frontals}, Theorem 3.8).
	The frontal map germ $f$ is a germ of wave front if and only if $\tilde f$ is an immersion, meaning that $\widetilde{f_t}$ is also an immersion and $f_t$ is a wave front for all $t \in U$.
	It follows that $\Delta_\ms{F}(f)$ is a wave front.
\end{remark}

\subsection{Generating families}
In their theory of deformations of wave fronts, Arnold and his colleagues proved that every germ of wave front $(X,0)$, $X \subseteq \mb{C}^{n+1}$ can be generated in terms of a smooth family of function germs
	\[g\colon (\mb{C}^k\times \mb{C}^r,0) \to (\mb{C},0)\]
with smooth critical set (see \cite{Arnold_I}, Chapter 19), which they called \emph{Morse families} of functions.
If one then considers the unfolding $G\colon (\mb{C}^k\times \mb{C}^r,0) \to (\mb{C}^k\times \mb{C},0)$ associated to this smooth family,
	\[G(u,x)=(u,g(u,x)),\]
then $(X,0)$ is equal to the germ at $0$ of the discriminant of $G$.

The original proof by Arnold and his colleagues is based on the fact that there is a correspondence between germs of caustics in $\mb{C}^n$ and germs of wave fronts in $\mb{C}^{n+1}$, sometimes called the \emph{symplectification functor}.
We now wish to give an alternative proof of this fact which does not require prior knowledge of symplectic topology and outlines a method to compute $G$ using a normalisation of $(X,0)$.

\begin{lemma}[Rank Distribution]\label{mu y k}
	Let $f\colon (\mb{C}^n,S) \to (\mb{C}^{n+1},0)$ be a germ of wave front of corank $k > 0$ in the form
		\[f(x,y)=(x,p_1(x,y),\dots,p_k(x,y),q(x,y)),\]
	with $x \in \mb{C}^{n-k}$ and $y \in \mb{C}^k$.
	If we denote the Nash lift of $f$ as
		\[\tilde f(x,y)=(f(x,y),\lambda_1(x,y),\dots,\lambda_{n-k}(x,y),\mu_1(x,y),\dots,\mu_k(x,y)),\]
	then the matrix
	\[\frac{\p \mu}{\p y}=
	\begin{pmatrix}
		\dfrac{\p \mu_1}{\p y_1} & \dots & \dfrac{\p \mu_1}{\p y_k} \\
		\vdots & \ddots & \vdots \\ 
		\dfrac{\p \mu_k}{\p y_1} & \dots & \dfrac{\p \mu_k}{\p y_k}
	\end{pmatrix}\]
	is invertible in a neighbourhood $U \subseteq \mb{C}^n$ of $0$.
\end{lemma}

\begin{proof}
	The mapping $f$ verifies the identity
		\[dq=\lambda_1\,dx_1+\dots+\lambda_{n-k}\,dx_{n-k}+\mu_1\,dp_1+\dots+\mu_k\,dp_k,\]
	from which follows that
	\begin{align*}
		\lambda_i=\frac{\p q}{\p x_i}-\mu_1\frac{\p p_1}{\p x_i}-\dots-\mu_k\frac{\p p_1}{\p x_k}; && \frac{\p q}{\p y_\ell}=\mu_1\frac{\p p_1}{\p y_\ell}+\dots+\mu_k\frac{\p p_k}{\p y_\ell}.
	\end{align*}
	for $i=1,\dots,n-k$ and $\ell=1,\dots,k$.
	Taking partial derivatives on both expressions yields
	\begin{align}
		\frac{\p \lambda_i}{\p y_\ell}=&\frac{\p^2 q}{\p y_\ell \p x_i}-\sum^k_{j=1}\left(\frac{\p \mu_j}{\p y_\ell}\frac{\p p_j}{\p x_i}+\mu_j\frac{\p^2 p_j}{\p y_\ell \p x_i}\right); \label{derivative lambda}\\
		 \frac{\p^2 q}{\p x_i\p y_\ell}=&\sum^k_{j=1}\left(\frac{\p\mu_j}{\p x_i}\frac{\p p_j}{\p y_\ell}+\mu_j\frac{\p^2 p_j}{\p x_i\p y_\ell}\right). \label{derivative q 2}
	\end{align}
	Combining \eqref{derivative lambda} and \eqref{derivative q 2}, we deduce that
	\begin{align} \label{derivative lambda 2}
		\frac{\p \lambda_i}{\p y_\ell}=&\sum^k_{j=1}\left(\frac{\p\mu_j}{\p x_i}\frac{\p p_j}{\p y_\ell}-\frac{\p \mu_j}{\p y_\ell}\frac{\p p_j}{\p x_i}\right).
	\end{align}
	Since $f$ has corank $k > 0$, we can choose coordinates in the source and target such that the $1$-jets of the functions $p_1,\dots,p_k$ vanish at the origin; this is, that
		\[\frac{\p p_j}{\p x_i}(0,0)=\frac{\p p_j}{\p y_\ell}(0,0)=0\]
	for $j,\ell=1,\dots,k$ and $i=1,\dots,n-k$.
	Therefore, it follows from \eqref{derivative lambda 2} that the functions $\lambda_1,\dots,\lambda_{n-k}$ have no linear term in $y$.
	
	Since $\tilde f$ is an immersion and $f$ has corank $k$, it follows that there exist $\mathbf{A} \in M_{k\times 1}(\ms{O}_{n-k})$, $\mathbf{B} \in M_{k\times k}(\ms{O}_n)$ such that $\mathbf{B}(0,0) \in \mathrm{GL}_k(\mb{C})$ and
	\[\begin{pmatrix}
		\mu_1(x,y)\\ \vdots \\ \mu_k(x,y)
	\end{pmatrix}
		=\mathbf{A}(x)+\mathbf{B}(x,y)
	\begin{pmatrix}
		y_1\\ \vdots \\ y_k
	\end{pmatrix}
	\]
	In particular, there is a $\mathbf{C} \in M_{k\times k}(\ms{O}_n)$ such that
	\[\frac{\p \mu}{\p y}(x,y)=\mathbf{B}(x,y)+\mathbf{C}(x,y)
		\begin{pmatrix}
			y_1\\ \vdots \\ y_k
		\end{pmatrix}
		\implies
		\det\left(\frac{\p \mu}{\p y}\right)(0,0)=\det(\mathbf{B})(0,0) \neq 0.
	\]
	Therefore, there is a neighbourhood $U \subseteq \mb{C}^n$ of $(0,0)$ such that $\det(\p \mu/\p y)(x,y)\neq 0$ for all $(x,y) \in U$.
\end{proof}

\begin{theorem}\label{generating family wave front}
	Let $(X,0)$ be a singular germ of wave front, $X \subseteq \mb{C}^{n+1}$.
	There exists a smooth map germ $h\colon (\mb{C}^N,T) \to (\mb{C}^{n+1},0)$, $N \geq n+1$, verifying the following conditions:
	\begin{enumerate}
		\item the set of critical points of $h$ is smooth;
		\item the discriminant of $h$ is equal to $(X,0)$.
	\end{enumerate}
	We call $h$ the \textbf{generating family} of $(X,0)$.
\end{theorem}

This result can be extended to regular germs $(X,0)$ by taking a parametrisation $f\colon (\mb{C}^n,S) \to (\mb{C}^{n+1},0)$.
This is done under the convention that the discriminant of a map germ $h\colon (\mb{C}^n,S) \to (\mb{C}^p,0)$ with $n < p$ is the image of $h$.

\begin{proof}[Proof of Theorem \ref{generating family wave front}]
	Let $f\colon (\mb{C}^n,S) \to (\mb{C}^{n+1},0)$ be a normalisation of $(X,0)$ with corank $k > 0$.
	We can choose coordinates in the source and target such that
		\[f(x,y)=(x,p_1(x,y),\dots,p_k(x,y),q(x,y)),\]
	with $x \in \mb{C}^{n-k}$ and $y \in \mb{C}^k$.
	We can further assume (by reordering the coordinates in the target if necessary) that
	\begin{equation} \label{frontal equation corank k}
		dq=\lambda_1\,dx_1+\dots+\lambda_{n-k}\,dx_{n-k}+\mu_1\,dp_1+\dots+\mu_k\,dp_k
	\end{equation}
	for some $\lambda_1,\dots,\lambda_{n-k},\mu_1,\dots,\mu_k \in \ms{O}_n$.
	By the Rank Distribution Lemma \ref{mu y k}, the matrix $\p \mu / \p y$ is invertible in an open neighbourhood $U \subseteq \mb{C}^k\times V$ of $S$. 
	
	Let $f\colon V \to \mb{C}^{n+1}$ be a representative of $f$.
	Consider the mapping $h\colon \mb{C}^k\times V \to \mb{C}^{n+1}$ given by
		\[h(u,x,y)=\left(x,u,q(x,y)+\sum_{i=1}^k \mu_i(x,y)(u_i-p_i(x,y))\right).\]
	Using Equation \eqref{frontal equation corank k}, the critical set of $h$ is given by the equations
		\[0=\frac{\p q}{\p y_j}+\sum_{i=1}^k  \frac{\p \mu_i}{\p y_j}(u_i-p_i)-\sum_{i=1}^k \mu_i\frac{\p p_i}{\p y_j}=\sum_{i=1}^k  \frac{\p \mu_i}{\p y_j}(u_i-p_i)\]
	for $j=1,\dots,k$.
	This is, we have the system of equations
		\[\begin{pmatrix}0\\\vdots\\0\end{pmatrix}
			=
			\begin{pmatrix}
				\dfrac{\p \mu_1}{\p y_1} & \dots & \dfrac{\p \mu_1}{\p y_k} \\
				\vdots & \ddots & \vdots \\ 
				\dfrac{\p \mu_k}{\p y_1} & \dots & \dfrac{\p \mu_k}{\p y_k}
			\end{pmatrix}
			\begin{pmatrix}
				u_1-p_1\\
				\vdots \\
				u_k-p_k
			\end{pmatrix}
			=
			\frac{\p \mu}{\p y}
			\begin{pmatrix}
				u_1-p_1\\
				\vdots \\
				u_k-p_k
			\end{pmatrix}
			.\]
	The matrix $\p \mu/\p y$ is invertible on $U$, meaning that $u_j=p_j(x,y)$ for $j=1,\dots,k$.
	This defines a graph-type submanifold of dimension $k$ on $U$, so the set of critical points of $h|_U$ is smooth.
	Moreover, we have that
		\[h(p(x,y),x,y)=(x,p(x,y),q(x,y))=f(x,y),\]
	so the discriminant of $h$ on $U$ is equal to $X\cap h(U)$.
\end{proof}

Note that the generating family $h\colon (\mb{C}^k\times \mb{C}^{n}, \{0\}\times S) \to (\mb{C}^{n+1},0)$ constructed in the proof of Theorem \ref{generating family wave front} has a representative in the form
	\[h(u,x,y)=(x,u,h'(u,x,y))\]
for some $h' \in \ms{O}_{n+k}$.
In Arnold's terminology, the generating family $h$ solely consists of the function $h'$, with $y \in \mb{C}^k$ as the variable and $x\in \mb{C}^{n-k}$ and $u \in \mb{C}^k$ as parameters.
Using Arnold's definition of the term \emph{generating family}, he showed the following:

\begin{theorem}[\cite{Arnold_I}, \S 21.4]\label{Arnold generating family stable}
	A germ of wave front $(X,0)$ with generating family $h' \in \ms{O}_N$ is stable if and only if $h'$ is $\ms{V}$-versal.
\end{theorem}

Theorem \ref{Arnold generating family stable} can be phrased in our language as follows: \emph{a germ of wave front $f\colon (\mb{C}^n,S) \to (\mb{C}^{n+1},0)$ with generating family $h\colon (\mb{C}^k\times\mb{C}^n,\{0\}\times S) \to (\mb{C}^{n+1},0)$ is $\ms{F}$-stable if and only if $h$ is $\ms{A}$-stable.}

\begin{remark}\label{Ak quasihomogeneous}
	Let $f\colon (\mb{C}^n,S) \to (\mb{C}^{n+1},0)$ be a $\ms{F}$-stable germ of corank $1$ wave fronts.
	By Theorem \ref{generating family wave front}, the image of $f$ is the discriminant of a corank $1$ smooth map germ $h\colon (\mb{C}^{n+1},0) \to (\mb{C}^{n+1},0)$.
	Since $f$ is $\ms{F}$-stable, $h$ is $\ms{A}$-stable by Theorem \ref{Arnold generating family stable}, which implies that $h$ is an $A_k$ singularity with $k \leq n+1$ (following the notation from Arnold's ADE classification \cite{Arnold_I}).
	A simple computation shows that the discriminant of the $A_k$ singularity is quasihomogeneous, so $f$ is a quasihomogeneous map germ.
\end{remark}

Using Theorem \ref{generating family wave front}, we can relate the frontal Milnor number of $f$ with the discriminant Milnor number of a generating family of $f$.
The discriminant Milnor number is an analytic invariant first described by Damon and Mond in \cite{MondDamon}, which counts the number of spheres emerging in the discriminant of a stable deformation of a smooth map germ $(\mb{C}^n,S) \to (\mb{C}^p,0)$ with $n \geq p$.

\begin{corollary}\label{frontal discriminant milnor}
	Let $f\colon (\mb{C}^n,S) \to (\mb{C}^{n+1},0)$ be an $\ms{F}$-finite, proper germ of wave fronts.
	Let $h\colon (\mb{C}^N,0) \to (\mb{C}^{n+1},0)$ be a generating family for the image of $f$.
	Then $f$ admits a frontal stabilisation if and only if $h$ admits a stabilisation.
	Moreover, when both integers are defined, we have the identity
		\[\mu_\ms{F}(f)=\mu_\Delta(h),\]
	where $\mu_\Delta(h)$ denotes the discriminant Milnor number of $h$.
\end{corollary}

\begin{proof}
	First assume that $f$ admits an $\ms{F}$-stabilisation $f_t$, with $t \in \mb{D}$.
	Let $t \in \mb{D}^*$ and $h_t$ be a generating family for $f_t$.
	For $t \neq 0$, frontal stability of $f_t$ implies $\ms{A}$-stability of $h_t$ due to Theorem \ref{Arnold generating family stable}.
	Therefore, $h$ admits a stabilisation.
	
	Conversely, assume that $h$ admits a stabilisation $h_t$, with $t\in \mb{D}$.
	The $1$-parameter unfolding $H'(t,y)=(t,h_t(y))$ defines a $1$-parameter frontal unfolding $F(t,x)=(t,f_t(x))$ of $f$.
	Since $h_t$ is stable for all $t \neq 0$, it follows from Theorem \ref{Arnold generating family stable} that $f_t$ is $\ms{F}$-stable for $t\neq 0$.
	Therefore, $f$ admits a frontal stabilisation.

	Now assume that $f$ admits a frontal stabilisation, so that $\mu_\ms{F}(f)$ is defined.
	By definition of frontal Milnor number, $\mu_{\ms{F}}(f)$ is the number of spheres in the image of $f_t$ for $t\in \mb{D}^*$.
	The image of $f$ is the discriminant of $h$, so the image of $f_t$ is the discriminant of the induced stabilisation $h_t$ and $\mu_{\ms{F}}(f)$ is equal to the number of spheres in a stable deformation of the discriminant of $h$, which is $\mu_\Delta(h)$.
\end{proof}

To finish this section, we have the following result below, which will prove useful in \S \ref{Disentanglement module}.
Let $(X,0)$ be an analytic hypersurface germ, with $X \subset \mb{C}^{n+1}$, and $g \in \ms{O}_{n+1}$ such that $(X,0)=(V(g),0)$.
We define $\Der(-\log X)$ as the sheaf of vector fields in $\mb{C}^n$ which are tangent to $X$ outside its singular locus.
This can be computed explicitly as
	\[\Der(-\log X)=\{\xi \in \theta_n: \xi(g) \in (g)\}.\]
We say that $(X,0)$ is a \textbf{free divisor} if $\Der(-\log X)$ is a $\ms{O}_{n+1}$-free module.

\begin{corollary}\label{wave fronts are free divisors}
	Every germ of stable wave front is a free divisor of rank $n+1$.
\end{corollary}

\begin{proof}
	Let $(X,0)$ be a germ of stable wave front.
	If $(X,0)$ is a smooth germ, then $(X,0)$ is a free divisor since $\Der(-\log (X,0))\cong \theta_n$, where $n$ is the dimension of $(X,0)$.
	Thus we can assume that $(X,0)$ has a singularity at the origin.
	
	By Theorem \ref{generating family wave front}, there exists a map germ $h\colon (\mb{C}^n,S) \to (\mb{C}^p,0)$ such that $(X,0)$ is the discriminant of $h$.
	Since $(X,0)$ is stable, $h$ is $\ms{A}$-stable by Theorem \ref{Arnold generating family stable}, and a result by Looijenga \cite{Looijenga_1984} implies that the discriminant of $h$ is a free divisor.
	Therefore, $(X,0)$ is a free divisor.
\end{proof}

\section{An expression for the frontal codimension} \label{Damon formula}
Let $(X,0)=(V(g),0)$ be an analytic hypersurface, with $X \subset \mb{C}^{n+1}$ and $g \in \ms{O}_{n+1}$.
We can consider the evaluation map for $g$,
\begin{funcion} \label{evaluation map}
	\ev\colon \Der(-\log X) \arrow[r] 	& J(g)	\\
	\xi \arrow[r, maps to] 			& \xi(g)
\end{funcion}
We denote the kernel of this map as $\Der(-\log g)$.

\begin{lemma}[\cite{Crk1Frontals}, cf. \cite{MondNuno}]\label{lift derlog}
	Let $f\colon (\mb{C}^n, S) \to (\mb{C}^{n+1}, 0)$ be a finite, proper frontal map germ with integral corank at most $1$, and $X$ be the germ at $0$ of the image of $f$.
	If $f$ is $\ms{F}$-finite and its Nash lift $\tilde f$ verifies that $\codim \Sigma(\tilde f) > 1$, 
		\[\Der(-\log X)=\{\eta \in \theta_{n+1}:  \omega f(\eta)=tf(\xi) \text{ for some }\xi \in \theta_n\}.\]
\end{lemma}

We first state a frontal version of Damon's Theorem from \cite{Damon}:

\begin{theorem}\label{Damon formula frontal codimension}
	Let $f$ be as above and $F$ be an $\ms{F}$-stable frontal unfolding of $f$, and $\ms{X}$ be the image of $F$.
	If $i\colon (\mb{C}^{n+1},0) \to (\mb{C}^{n+1}\times \mb{C}^r,0)$ is the inclusion germ
		\[\frac{\ms{F}(f)}{T\ms{A}_ef}\cong \frac{\theta(i)}{ti(\theta_{n+1})+i^*\Der(-\log \ms{X})}.\]
\end{theorem}

To prove this result, we first need to introduce some notation.
Consider the inclusion map germ $j_n\colon (\mb{C}^r,0) \to (\mb{C}^r\times \mb{C}^n,0)$ given by $u \mapsto (u,0)$.
We define $\theta_{n+r/r}$ (resp. $\theta_{n+r+1/r}$) as the module of vector fields $\xi \in \theta_{n+r}$ (resp. $\xi \in \theta_{n+r+1})$ such that $j_n^*\xi$ (resp. $j_{n+1}^*\xi$) is constant.
This can be interpreted as saying that vector field germs on these modules do not depend on the variables of the parameter space.
We call these \emph{relative} modules.

We also set
\begin{align*}
	\ms{F}(F/r)=\ms{F}(F)\cap \theta_{n+r/r}; && \Der(-\log \ms{X}/r)=\Der(-\log \ms{X})\cap \theta_{n+r+1/r}
\end{align*}
as the relative versions of $\ms{F}(F)$ and $\Der(-\log \ms{X})$.

\begin{proof}
We follow the proof from \cite{MondNuno}.
Let $\pi\colon \mb{C}^r\times \mb{C}^{n+1}\to \mb{C}^r$ be the canonical projection.
We consider the commutative diagram

\begin{diagram*}
	&			& 0			 							& 			0 								& 															&	\\
C:	&0 \arrow[r] 	& t\pi(\Der(-\log \ms{X})) \arrow[u] \arrow[r]		& \theta(\pi)	\arrow[u] \arrow[r] 						& 0															&	\\
B:	&0 \arrow[r] 	& \Der(-\log \ms{X}) \arrow[u,"t\pi"] \arrow[r, hook]	& \theta_{n+r+1}	\arrow[u,"t\pi"] \arrow[r,"\omega F"] 			& \displaystyle \frac{\ms{F}(F)}{tF(\theta_{n+r})} \arrow[r] \arrow[u] 			& 0	\\
A:	&0 \arrow[r] 	& \Der(-\log \ms{X}/r) \arrow[u, hook] \arrow[r, hook]	& \theta_{n+r+1/r}	\arrow[u, hook] \arrow[r, "\omega F"] 	& \displaystyle \frac{\ms{F}(F/r)}{tF(\theta_{n+r/r})} \arrow[r] \arrow[u] 	& 0	\\
	&			& 0 \arrow[u]		 					& 0 \arrow[u] 									& 0 \arrow[u]													&
\end{diagram*}

Each one of the rows in this diagram is a complex, which we label as $A$, $B$ and $C$.
In particular, we note that complex $B$ is an exact sequence, as a consequence of Lemma \ref{lift derlog} and the frontal stability of $F$.
Furthermore, the complexes in the diagram above form a short exact sequence
\begin{diagram} \label{chain complexes}
	0 \arrow[r] & A \arrow[r, hook] & B \arrow[r, "t\pi"] & C \arrow[r] & 0,
\end{diagram}
which we can extend into a long exact sequence in homology,
\begin{diagram}
	\dots \arrow[r] & H_1(B) \arrow[r] & H_1(C) \arrow[r] & H_0(A) \arrow[r] & H_0(B) \arrow[r] & \dots
\end{diagram}
Exactness of the complex $B$ implies that $H_*(B)=0$, so $H_1(C) \cong H_0(A)$.

The boundary operators in the complex $A$ are the inclusion and $\omega F$, so we have
	\[H_0(A)=\frac{\ms{F}(F/r)}{tF(\theta_{n+r/r})+\omega F(\theta_{n+r+1/r})};\]
the boundary operator in the complex $C$ is the inclusion, so
	\[H_1(C)=\frac{\theta(\pi)}{t\pi(\Der(-\log \ms{X}))}.\]
Dividing on the LHS by $\mf{m}_r\theta(F/r)\cap\ms{F}(F/r)$ and on the RHS by $\mf{m}_r\theta(\pi)$, we get
\begin{multline*}
	\frac{\theta(\pi)}{t\pi(\Der(-\log \ms{X}))+\mf{m}_d\theta(\pi)}\cong \\
	\cong \frac{\ms{F}(F/d)}{tF(\theta_{n+r/r})+\omega F(\theta_{n+r+1/r})+\mf{m}_r\theta(F/r)\cap\ms{F}(F/r)}
	\cong \frac{\ms{F}(f)}{tf(\theta_n)+\omega f(\theta_{n+1})}.
\end{multline*}
To finish the proof, we shall rewrite the left-hand side in terms of the $i$ mapping.
To do this, note that
\begin{align*}
	\theta(\pi)=\sum^r_{i=1}\ms{O}_{n+r+1} \frac{\p}{\p u_i} &\implies \frac{\theta(\pi)}{\mf{m}_r\theta(\pi)}=\sum^r_{i=1}\ms{O}_{n+1} \frac{\p}{\p u_i};\\
	\theta(i)=\sum^{n+1}_{i=1}\ms{O}_{n+1} \frac{\p}{\p y_i} \oplus \sum^r_{j=1}\ms{O}_{n+1} \frac{\p}{\p u_j} &\implies \frac{\theta(i)}{ti(\theta_{n+1})}=\sum^r_{i=1}\ms{O}_{n+1} \frac{\p}{\p u_i};
\end{align*}
It follows that
	\[\frac{\theta(\pi)}{t\pi(\Der(-\log \ms{X}))+\mf{m}_r\theta(\pi)}\cong \frac{\theta(i)}{ti(\theta_{n+1})+i^*\Der(-\log \ms{X})}.\]
\end{proof}

\begin{corollary}\label{formula codimension wave front}
	Let $f\colon (\mb{C}^n,S) \to (\mb{C}^{n+1},0)$ be a finite germ of wave front and $h\colon (\mb{C}^N,0) \to (\mb{C}^{n+1},0)$ be the generating family of $f$ given in the proof of Theorem \ref{generating family wave front}.
	Then,
		\[\frac{\ms{F}(f)}{T\ms{A}_ef}\cong \frac{\theta(h)}{T\ms{A}_eh}.\]
	In particular, $\codim_{\ms{F}_e} f=\codim_{\ms{A}_e} h$.
\end{corollary}

\begin{proof}
	First note that $f$ has integral corank $0$ since its Nash lift $\tilde f$ is an immersion.
	This also implies that $\Sigma(\tilde f)$ is equal to the empty set as germs at $S$, thus it has codimension greater than $1$.

	Let $F\colon (\mb{C}^r\times \mb{C}^n,\{0\}\times S) \to (\mb{C}^r\times \mb{C}^{n+1},0)$ be a $\ms{F}$-stable frontal unfolding of $f$.
	Following the arguments in Remark \ref{conjecture Legendrian}, $F$ is a germ of wave front.
	Using Theorem \ref{Damon formula frontal codimension}, we have that
		\[\frac{\ms{F}(f)}{T\ms{A}_ef}\cong \frac{\theta(i)}{ti(\theta_{n+1})+i^*\Der(-\log \ms{X})},\]
	where $\ms{X}$ is the image of $F$.
	
	Let $H\colon (\mb{C}^{N'},0) \to (\mb{C}^r\times \mb{C}^{n+1},0)$ be the generating family of $F$ given in the proof of Theorem \ref{generating family wave front}.
	It is clear by construction that $H$ is an unfolding of $h$ with discriminant $\ms{X}$.
	Moreover, $H$ is a stable map germ by Theorem \ref{Arnold generating family stable}.
	It then follows from Damon's Theorem \cite{Damon} that
		\[\frac{\theta(i)}{ti(\theta_{n+1})+i^*\Der(-\log \ms{X})} \cong \frac{\theta(h)}{T\ms{A}_eh}.\]
	Combining both identities, we have the statement.
\end{proof}

A result by Damon and Mond \cite{MondDamon} states that, for a smooth, $\ms{A}$-finite map germ $h \colon (\mb{C}^q,S) \to (\mb{C}^p,0)$ such that $(q,p)$ are nice dimensions, we have the inequality
	\[\codim_{\ms{A}_e} h \leq \mu_\Delta(h)\]
with equality if $h$ is quasihomogeneous.

Let $f\colon (\mb{C}^n,S) \to (\mb{C}^{n+1},0)$ be a finite, $\ms{F}$-finite germ of wave front, and $h\colon (\mb{C}^N,0) \to (\mb{C}^{n+1},0)$ be the generating family of $f$ given in the proof of Theorem \ref{generating family wave front}.
We have that $N=n+k$, where $k$ is the corank of $f$.
Using Corollary \ref{frontal discriminant milnor}, we have $\codim_{\ms{F}_e} f=\codim_{\ms{A}_e} h$.

If $k=1$, then $h\colon (\mb{C} \times \mb{C}^n, \{0\}\times S) \to (\mb{C}^{n+1},0)$ is a corank $1$ equidimensional map germ and it admits a stabilisation.
Conversely if $k > 1$ but $(n+k,n+1)$ are in Mather's nice dimensions, $h$ admits a stabilisation.
In both cases, it follows from Corollary \ref{formula codimension wave front} that $f$ admits a frontal stabilisation and $\mu_{\ms{F}}(f)=\mu_\Delta(h)$.
Therefore, if $(n+k,n+1)$ are in the nice dimensions or $k=1$, we conclude that
\begin{equation} \label{mond conjecture wave front}
	\codim_{\ms{F}_e} f \leq \mu_\ms{F}(f),
\end{equation}
with equality if $f$ is quasihomogeneous.

In \S \ref{Disentanglement module}, we shall use a different approach to prove this result, which does not require the use of generating families.
The methods developed in that section may also offer a glimpse into how to prove the conjecture in the case of corank $1$ proper frontal map germs with $n=2$.

\begin{definition}
	Let $f\colon (\mb{C}^n,S) \to (\mb{C}^{n+1},0)$ be a smooth map germ with image equation $g=0$.
	Let $F$ be an $r$-parameter unfolding of $f$ and $j\colon (\mb{C}^{n+1},0) \to (\mb{C}^{n+1}\times \mb{C}^r,0)$ be the inclusion germ given by $y \mapsto (y,0)$.
	We say that $G=0$ is a \textbf{good defining equation} for the image of $F$ if it is reduced, $G \in J(G)$ and $g=j^*G$.
\end{definition}

We can always construct a frontal unfolding admitting a good defining equation.
Indeed, if $F(x,u)=(f_u(x),u)$ is a frontal unfolding of $f$ with reduced image equation $G$, the unfolding
	\[\tilde F(x,u,t)=(f_u(x),u,t)\]
is trivially frontal with reduced image equation $\tilde G(y,u,t)=G(y,u)e^t$.

If $G$ is a good defining equation, the evaluation map \eqref{evaluation map} is surjective and we have an isomorphism
\begin{equation} \label{iso derlog}
	\frac{\theta_{n+r+1}}{\Der(-\log G)} \longrightarrow J(G).
\end{equation}
Moreover, since $G \in J(G)$, there exists a $\epsilon \in \theta_{n+r+1}$ such that $\ev(\epsilon)=G$; this vector field germ satisfies the identity
	\[\Der(-\log \ms{X})=(\epsilon) \oplus \Der(-\log G).\]
If we assume that $G(y,0)=g(y)$, the evaluation map sends $\Der(-\log \ms{X})$ to $(G)$.
We also define $J_y(G)$ as the image under the evaluation map of $\p y_1,\dots,\p y_{n+1}$,
	\[J_y(G)=\left(\frac{\p G}{\p y_1},\dots,\frac{\p G}{\p y_{n+1}}\right).\]
Therefore, the isomorphism above induces in turn another isomorphism,
\begin{equation} \label{isomorphism pre tensor product}
	\frac{\theta_{n+r+1}}{\Der(-\log \ms{X})+\ms{O}_{n+r+1}\{\p y_1,\dots, \p y_{n+1}\}}\cong \frac{J(G)}{J_y(G)+(G)}.
\end{equation}
Tensoring on both sides of the identity with the ring $\ms{O}_r/\mf{m}_r$ yields
	\[\frac{\theta(i)}{i^*\Der(-\log \ms{X})+ti(\theta_{n+1})}\cong \frac{J(G)}{J_y(G)+(G)}\otimes_{\ms{O}_r} \frac{\ms{O}_r}{\mf{m}_r}.\]
Using Theorem \ref{Damon formula frontal codimension}, we conclude that
	\[\frac{\ms{F}(f)}{tf(\theta_n)+\omega f(\theta_{n+1})}\cong \frac{J(G)}{J_y(G)+(G)}\otimes_{\ms{O}_r} \frac{\ms{O}_r}{\mf{m}_r},\]

\begin{corollary}\label{frontal codimension good equation}
	Let $f$, $F$ and $G$ be as above.
	We have the identity
		\[\codim_{\ms{F}_e}(f)=\dim_{\mb{C}}\left(\frac{J(G)}{J_y(G)+(G)}\otimes_{\ms{O}_r} \frac{\ms{O}_r}{\mf{m}_r}\right)\]
\end{corollary}

Note that we need to assume that $\Sigma(\tilde f)$ has codimension greater than $1$ for either of the theorems stated in this section to hold.
In particular, this result always holds for wave fronts.

\section{A Jacobian module for frontal disentanglements} \label{Disentanglement module}
The expression for the frontal codimension derived on the previous section is based on a similar result from \cite{Disentanglements}.
In that article, the authors define a module $M(g)$ for a smooth map germ $f\colon (\mb{C}^n,S) \to (\mb{C}^{n+1},0)$ with reduced image equation $g=0$, whose dimension over $\mb{C}$ equals $\mu_I(f)$ when $(n,n+1)$ are in the range of Mather's nice dimensions (i.e. for $n \leq 14$) or when $f$ has corank $1$.

Given a smooth $r$-parameter unfolding $F$ of $f$, they also construct a \emph{relative} module $M_y(G)$ with the property that
	\[M_y(G) \otimes_{\ms{O}_r} \frac{\ms{O}_r}{\mf{m}_r} \cong M(g),\]
where $G=0$ is a reduced image equation for $F$ which specifies into $g$.
They then go to show that under the conditions above, $\mu_I(f)\geq \codim_{\ms{A}_e}f$ if and only if $M_y(G)$ is Cohen-Macaulay of dimension $r$.
In particular, they show that $M(g)$ has finite dimension over $\mb{C}$ if and only if $f$ is $\ms{A}$-finite.

The main obstacle one faces when trying to reach a similar result in the framework of frontal mappings is that some of the key results in \cite{Disentanglements} rely on the fact that the singular set of an $\ms{A}$-finite map germs $(\mb{C}^n,S) \to (\mb{C}^{n+1},0)$ has codimension greater than $1$ when $n > 1$.
Nonetheless, if the singular set of $f$ has codimension $1$, the module $M(g)$ can sometimes have finite dimension over $\mb{C}$ even if $f$ is not $\ms{A}$-finite (see Example \ref{cuspidal Mg} below).
This motivates the question of whether the results from \cite{Disentanglements} can be adapted to our framework.

We begin by giving some preliminary definitions.
The conductor ideal of a smooth map germ $f\colon (\mb{C}^n,S) \to (\mb{C}^{n+1},0)$ with image germ $(X,0)$ is given by
	\[C(f)=\{h \in \ms{O}_{X,0}: h \cdot \ms{O}_n \subset \ms{O}_{X,0}\}.\]
A lemma of Piene \cite{Piene} gives an expression to compute this ideal:

\begin{lemma} \label{piene lemma}
	Let $g(y)=0$ be a reduced equation for the image of $f$.
	If $f$ is finite and generically 1-to-1, there is a unique $\lambda \in \ms{O}_n$ such that
		\[\frac{\p g}{\p y_i}\circ f = (-1)^i\lambda \det(df_1,\dots, \widehat{df_i}, \dots, df_{n+1})\]
	for $i=1,\dots,n+1$, where $\widehat{\cdot}$ denotes an omitted entry.
	Moreover, $C(f)$ is generated by $\lambda$.
\end{lemma}

It follows from the proof of \cite{MondPellikaan}, Theorem 3.4 that the preimage of $C(f)$ via $f^*$ is the first Fitting ideal of $f$.
Moreover, multiplication by $\lambda$ induces an isomorphism
	\[\frac{C(f)}{J(g)\ms{O}_n}\cong \frac{\ms{O}_n}{\mc{R}(f)},\]
where $\mc{R}(f)$ is the ramification ideal of $f$ (see \S 1).
It is worth noting that, unlike in the $\ms{A}$-finite case, the ring $\ms{O}_n/\mc{R}(f)$ has codimension $1$ in $\ms{O}_n$ when $f$ is a frontal map germ, so it is not determinantal.

\begin{definition}
	Let $f\colon (\mb{C}^n,S) \to (\mb{C}^{n+1},0)$ be a finite smooth map germ with image $(X,0)=(V(g),0)$, $g$ reduced.
	We define the module $\hat M(g)$ as the kernel of the epimorphism
		\[\dfrac{\ms{F}_1(f)}{J(g)} \longrightarrow \dfrac{C(f)}{J(g)\ms{O}_n}\]
	induced by $f^*$, where $J(g)\ms{O}_n$ is the ideal in $\ms{O}_n$ generated by $f^*(J(g))$.
	Moreover since $(f^*)^{-1}(C(f))=\ms{F}_1(f)$,
		\[\hat M(g)=\frac{(f^*)^{-1}(J(g)\ms{O}_n)}{J(g)}.\]
\end{definition}

\begin{example}\label{cuspidal Mg}
	Let $f\colon (\mb{C}^2,0) \to (\mb{C}^3,0)$ be the cuspidal edge, parametrised as
		\[f(x,y)=(x,y^2,y^3).\]
	This mapping is $\ms{F}$-stable and its singular set is given by the source line of equation $y=0$.
	Moreover, its frontal Milnor number is $0$, as its image is homeomorphic to a plane.
	
	Let $g(X,Y,Z)=Z^2-Y^3$, so that $(V(g),0)$ is the germ of the image of $f$.
	We have
		\[\dfrac{\ms{F}_1(f)}{J(g)}=\frac{\langle Z,Y\rangle }{\langle Z,Y^2\rangle} \cong \frac{\langle Y\rangle }{\langle Y^2\rangle} \cong \frac{\langle y^2 \rangle}{\langle y^3 \rangle}=\dfrac{C(f)}{J(g)\ms{O}_2},\]
	meaning that $\hat M(g)=0$.
	Therefore, $\dim_\mb{C} \hat M(g)$ is finite. 
\end{example}

The reason to denote this module as $\hat M(g)$ instead of simply $M(g)$ is that the dimension of $\hat M(g)$ does not coincide in general with the frontal Milnor number (see Example \ref{Mg whitney umbrella}).

\begin{proposition}[\cite{Disentanglements}, Proposition 3.3] \label{snake lemma Mg}
	Assume that $f$ is finite and generically 1-to-1.
	We have the following exact sequence of $\ms{O}_{n+1}$-modules:
	\begin{diagram*}
		0 \arrow[r] & \dfrac{J(g)+\langle g\rangle}{J(g)} \arrow[r] & \hat M(g) \arrow[r] & \dfrac{J(g)\ms{O}_n}{J(g)\ms{O}_{X,0}} \arrow[r] & 0,
	\end{diagram*}
	where $(X,0)=(V(g),0)$.
\end{proposition}

Proposition \ref{snake lemma Mg} is proved on \cite{Disentanglements} by constructing a commutative diagram to apply the snake lemma.
This diagram can be constructed so long as Lemma \ref{piene lemma} holds for $f$, hence the repetition in the assumptions.
In particular, if $f$ is finitely determined, the module on the right-hand side is isomorphic over $\mb{C}$ to $\theta(f)/T\ms{A}_ef$ (\cite{Mond_VanishingCycles}, Proposition 2.1), and $\hat M(g)$ can be related to the $\ms{A}_e$-codimension of $f$.
This is no longer true when $f$ is not finitely determined, as shown in the example below.

\begin{example}\label{Mg whitney umbrella}
	Let $f\colon (\mb{C}^2,0) \to (\mb{C}^3,0)$ be the folded Whitney umbrella, parametrised as
		\[f(x,y)=(x,y^2,xy^3).\]
	This map germ is not finitely determined, hence its $\ms{A}_e$-codimension is infinite.
	However, $J(g)\ms{O}_2/J(g)\ms{O}_{X,0}$ is generated by the class of the function $xy^4$ modulo $J(g)\ms{O}_{X,0}$.
	It follows that
		\[\dim_\mb{C} \frac{J(g)\ms{O}_2}{J(g)\ms{O}_{X,0}}=\dim_\mb{C} \hat M(g)=1.\]
	This is not equal to the frontal Milnor number of $f$, which is $0$.
\end{example}

Now let $F\colon (\mb{C}^r\times \mb{C}^n, \{0\}\times S) \to (\mb{C}^r\times \mb{C}^{n+1},0)$ be a $\ms{F}$-stable frontal unfolding of $f$, and $G=0$ be a good defining equation for the image of $F$.
If $j\colon (\mb{C}^{n+1},0) \to (\mb{C}^r\times \mb{C}^{n+1},0)$ is the germ of the inclusion $y \mapsto (0,y)$, then we have $J(g)=j^*(J_y(G))$.

\begin{definition}\label{Mg definition}
	Let $f\colon (\mb{C}^n,S) \to (\mb{C}^{n+1},0)$ be a proper frontal map germ with integral corank at most $1$ and image equation $g=0$.
	Let $F$ be a frontal stable $r$-parameter unfolding of $f$ with good defining equation $G=0$.
	We define the modules
	\begin{align*}
		M_y(G)=\frac{J(G)}{J_y(G)}; && M_{\ms{F}}(g)=M_y(G)\otimes_{\ms{O}_r} \frac{\ms{O}_r}{\mf{m}_r}.
	\end{align*}
\end{definition}

The assumptions from Definition \ref{Mg definition} will be held throughout the rest of this section.

\begin{example}
	Let $f(x,y)=(x,y^2,y^7+x^7y^5)$.
	A simple computation on \textsc{Singular} shows that $M_{\ms{F}}(g)$ is not finite. 
\end{example}

Note that the module $M_y(G)$ has slicing dimension $r$ in $\ms{O}_{n+r+1}$ by construction.
Since $M_y(G)$ is finitely generated and $\ms{O}_{n+r+1}$ is Noetherian, it has Krull dimension $r$.

We wish to show that the dimension of $M_{\ms{F}}(g)$ is equal to $\mu_{\ms{F}}(f)$.
In order to do this, we shall make use of the Samuel multiplicity and the following results:

\begin{lemma}\label{conservation of multplicity}
	Let $M$ be a finitely generated sheaf of $\ms{O}_n$-modules with representative $\ms{M}$.
	There exists an $\eps > 0$ such that
		\[e(\mf{m}_r,M)=\sum_{p \in B_\eps} e(\mf{m}_{r,w},\ms{M}_{(p,w)})\]
	where $B_\eps \subseteq \mb{C}^r$ denotes the ball of centre $0$ and radius $\eps$.
\end{lemma}

This result can be derived from \cite{MondNuno}, Corollary E.5 by taking the canonical projection $\mb{C}^{n+1} \times \mb{C}^r \to \mb{C}^r$.
Of course, there is a clear caveat that this sum will only make sense when almost all the multiplicities on the RHS are zero; however, this assumption will be shown to hold in the context where we wish to apply it.

In our proof, we shall consider the sheaf of $\ms{O}_n$-modules $M_{rel,\ms{F}}(G)$ given by
\begin{align*}
	M_y(G)_{(p,w)}=\frac{J(G)_{(p,w)}}{J_y(G)_{(p,w)}},
\end{align*}
where $M_y(G)$ is supported on an open neighbourhood $V\times W \subseteq \mb{C}^{n+1}\times \mb{C}^r$ of $0$.
We shall denote the representative of a given stalk as $\ms{M}_y(F)_{(p,w)}$.

\begin{lemma} \label{multiplicity inequality}
	Given a module $M$,
		\[e(\mf{m}_r,M) \leq \dim_\mb{C} \left(M\otimes_{\ms{O}_r} \frac{\ms{O}_r}{\mf{m}_r}\right)\]
	with equality if $M$ is Cohen-Macaulay.
\end{lemma}

\begin{theorem} \label{frontal Milnor multiplicity}
	Let $f\colon (\mb{C}^n,S) \to (\mb{C}^{n+1},0)$ be a proper $\ms{F}$-finite germ of wave fronts.
	Assume that either $(n,n+1)$ are frontal nice dimensions or $f$ has corank $1$.
	Let $F\colon (\mb{C}^r\times \mb{C}^n,\{0\}\times S) \to (\mb{C}^r\times \mb{C}^{n+1},0)$ be a stable frontal unfolding of $f$ and $G=0$ be a good defining equation for the image of $F$.
	Then, we have
		\[\mu_\ms{F}(f)=e(\mf{m}_r,M_y(G)).\]
\end{theorem}

\begin{proof}
	Let $F\colon U \to V\times W$ be a representative of $f$.
	We set $F(x,w)=(f_w(x),w)$ with $f_w\colon U_w \to V_w$, and write $g_w=0$ for the reduced image equation of $f_w$.

	Using a result by Siersma \cite{Siersma}, we can compute the frontal Milnor number of $f$ using the expression
	\begin{equation} \label{Siersma formula v1}
		\mu_\ms{F}(f)=\sum_{p \in B_\eps}\mu(g_w;p)=\sum_{p \in B_\eps}\dim_\mb{C} \frac{\ms{O}_{B_\eps,p}}{J(g_w)_p},
	\end{equation}
	for a generic $w \in W$, where $B_\eps \subseteq V_w$ is a Milnor ball of radius $\eps > 0$ centered at $p$.
	Since $f_w$ is stable for $w \neq 0$, $g_w$ has an isolated singularity at $w=0$.
	This implies that $\ms{O}_{B_\eps,p}=J(G)_{(p,w)}$ for a generic $w$ and we have the identity
	\begin{equation}\label{tensor product for Jacobian algebra}
		\frac{\ms{O}_{B_\eps,p}}{J(g_w)_p} \cong \frac{J(G)_{(p,w)}}{J_y(G)_{(p,w)}}\otimes_{\ms{O}_{r,w}}\frac{\ms{O}_{r,w}}{\mf{m}_{r,w}}=M_{\ms{F}}(g)_{p}.
	\end{equation}
	
	Since $f$ is a germ of wave fronts, $f_w$ is a $r$-parameter family of wave fronts by Remark \ref{conjecture Legendrian}.
	If $(n,n+1)$ are frontal nice dimensions, then $n \leq 5$ and $f_w$ belongs to one of the families in Arnold's ADE classification of stable wave fronts \cite{Arnold_I} for all $w \neq 0$, meaning that $f_w$ is quasihomogeneous.
	If $f$ has corank $1$, $f_w$ is an $A_\ell$ singularity with $\ell \leq n+1$ for $w \neq 0$ (see Remark \ref{Ak quasihomogeneous}) and it is quasihomogeneous.
	In both cases, the RHS in Equation \eqref{tensor product for Jacobian algebra} becomes the zero module along $V_w$.
	Therefore, Equation \eqref{Siersma formula v1} becomes
	\begin{equation} \label{Siersma formula v2}
		\mu_\ms{F}(f)=\sum_{p \in B_\eps\backslash V_w}\dim_\mb{C} \left(\ms{M}_y(G)_{(p,w)}\otimes_{\ms{O}_{r,w}} \frac{\ms{O}_{r,w}}{\mf{m}_{r,w}}\right)
	\end{equation}

	Let $\ms{M}_y(G)_{(p,w)}$ be a representative of $M_y(G)_{(p,w)}$.
	By Lemma \ref{multiplicity inequality}, we have the inequality
		\[\dim_\mb{C} \left(\ms{M}_y(G)_{(p,w)}\otimes_{\ms{O}_{r,w}} \frac{\ms{O}_{r,w}}{\mf{m}_{r,w}}\right) \geq e(\mf{m}_{r,w},\ms{M}_y(G)_{(p,w)}).\]
	We wish to show that we have equality for $w \neq 0$.
	To do so, note that the inclusion map germ $(\mb{C}^r,w) \rightarrow (\mb{C}^r\times B_\eps,(p,w))$ induces a module epimorphism $\ms{M}_y(G)_{(p,w)} \to \ms{O}_{r,w}$, hence the dimension of $\ms{M}_y(G)_{(p,w)}$ is at least $r$.
	Additionally, since $g_w$ is regular for $w \neq 0$, the module
		\[\ms{M}_y(G)_{(p,w)}\otimes_{\ms{O}_{r,w}} \frac{\ms{O}_{r,w}}{\mf{m}_{r,w}} \cong \frac{\ms{O}_{B_\eps,p}}{J(g_w)}\]
	has finite length and thus dimension $0$.
	It follows that $\ms{M}_y(G)_{(p,w)}$ has dimension $r$, so it is a complete intersection ring and thus a Cohen-Macaulay $\ms{O}_r$-module.
	It then follows from Lemma \ref{multiplicity inequality} that we achieve the equality and equation \eqref{Siersma formula v2} becomes
		\[\mu_\ms{F}(f)=\sum_{\mathclap{p \in B_\eps\backslash V_w}} e(\mf{m}_{r,w},\ms{M}_y(G)_{(p,w)}).\]
	Choosing a finitely supported representative $\ms{M}_y(G)_{(p,w)}$, we now conclude from Lemma \ref{conservation of multplicity} that
		\[\mu_\ms{F}(f)=e(\mf{m}_r,M_y(G)),\]
	as stated.
\end{proof}

\begin{theorem}\label{mond conjecture wave front}
	Let $f\colon (\mb{C}^n,S) \to (\mb{C}^{n+1},0)$ be a $\ms{F}$-finite germ of wave front.
	Assume that $f$ is finite and either $(n,n+1)$ are frontal nice dimensions or $f$ has corank $1$.
	Then,
		\[\mu_{\ms{F}}(f) \geq \codim_{\ms{F}_e}f,\]
	with equality if $f$ is quasihomogeneous.
\end{theorem}

To prove this result, we need a result by Pellikaan:

\begin{lemma}[\cite{Pellikaan}, \S 3] \label{Pellikaan lemma}
	Let $J\subset I$ be ideals in a ring $R$.
	Assume that $J$ is generated by $\codim I/J$ elements and $R/I$ is a Cohen-Macaulay module of codimension $2$.
	Then, $I/J$ is a Cohen-Macaulay module.
\end{lemma}

\begin{proof}[Proof of Theorem \ref{mond conjecture wave front}]
	Let $F, G$ be as in Theorem \ref{frontal Milnor multiplicity}.
	Corollary \ref{frontal codimension good equation} states that
	\begin{multline} \label{inequality codimension milnor}
		\codim_{\ms{F}_e}(f)=\dim_{\mb{C}}\left(\frac{J(G)}{J_y(G)+(G)}\otimes_{\ms{O}_r} \frac{\ms{O}_r}{\mf{m}_r}\right) \leq \\ \leq \dim_{\mb{C}}\left(\frac{J(G)}{J_y(G)}\otimes_{\ms{O}_r} \frac{\ms{O}_r}{\mf{m}_r}\right)=\dim_{\mb{C}}M_\ms{F}(g).
	\end{multline}
	On the other hand, using Lemma \ref{multiplicity inequality} and Theorem \ref{frontal Milnor multiplicity}, we have that
	\begin{equation} \label{eq: inequality multiplicity}
		\dim_{\mb{C}}M_{\ms{F}}(g) \geq e(\mf{m}_e,M_y(G))=\mu_\ms{F}(f).
	\end{equation}
	If we can show that $M_y(G)$ is a Cohen-Macaulay module, it will follow from Lemma \ref{multiplicity inequality} that the inequality \eqref{eq: inequality multiplicity} is an equality.

	Let $\ms{X}$ be the image of $F$.
	Using Corollary \ref{wave fronts are free divisors}, we have that $(X,0)$ is a free divisor, meaning that $\Der(-\log \ms{X})$ is a free module over $\ms{O}_{n+r+1}$.
	Since $G$ is a good defining equation, we have the decomposition
		\[\Der(-\log \ms{X})=\langle \epsilon \rangle \oplus \Der(-\log G),\]
	where $\epsilon \in \theta_{n+r+1}$ is such that $\epsilon(G)=G$, so $\Der(-\log G)$ is also a free module over $\ms{O}_{n+r+1}$.
	We can then construct the following free resolution:
	\begin{funcion*}
		0 \arrow[r] &[-1em] \Der(-\log G)  \arrow[r] 	& \theta_{n+r+1} \arrow[r, "\ev_G"] 	& \ms{O}_{n+r+1} \arrow[r] 	& \dfrac{\ms{O}_{n+r+1}}{J(G)}  \arrow[r] 	&[-1em] 0 \\
				&					& \xi \arrow[r, maps to] 			& \xi(G)					& 									& 
	\end{funcion*}
	It follows that the projective dimension of $\ms{O}_{n+r+1}/J(G)$ is $2$.
	Using the Auslander-Buchsbaum formula, we deduce that the depth of $\ms{O}_{n+r+1}/J(G)$ is equal to $n+r-1$.
	
	Using Corollary \ref{wave fronts are free divisors} once again, we have that $G$ is the generating equation of the singular locus of a smooth map germ $(\mb{C}^{N+r},0) \to (\mb{C}^{n+r+1},0)$.
	It follows that $\ms{O}_{n+r+1}/J(G)$ has dimension $n+r-1$, and thus it is a Cohen-Macaulay module of codimension $2$.
	
	Recall that $J_y(G)$ is spanned by definition by the functions
		\[\frac{\p G}{\p y_1}, \dots, \frac{\p G}{\p y_{n+1}},\]
	meaning that the slicing dimension of $J(G)/J_y(G)$ is equal to $r$.
	Therefore, the number of generators of $J_y(G)$ is equal to the codimension of $J(G)/J_y(G)$, so $M_y(G)$ is a Cohen-Macaulay module by Lemma \ref{Pellikaan lemma}.
	We conclude using Lemma \ref{multiplicity inequality} that Equation \eqref{eq: inequality multiplicity} is an equality.
	
	If $f$ is quasihomogeneous, $G$ is quasihomogeneous and we attain the equality in Equation \eqref{inequality codimension milnor} due to Lemma \ref{multiplicity inequality}.
\end{proof}

\section{Declarations}
The authors have no competing interests to declare that are relevant to the content of this article.

\end{document}